\documentclass[11pt,leqno]{article}
\usepackage[T1]{fontenc}
\usepackage{lipsum}
\usepackage{eso-pic}
\usepackage[utf8]{inputenc}
\usepackage[english]{babel}
\usepackage{tikz}
	
\usepackage{graphicx}
\usetikzlibrary{calc}
\usetikzlibrary{intersections}
\usepackage{verbatim}
\usepackage{etoolbox}
\usepackage{tkz-euclide}
\usepackage{geometry}
\usepackage{mathtools}
\usepackage{color}
\usepackage[title]{appendix}
\usepackage[all]{xy}
\usepackage{tikz-cd}
\usepackage{blindtext}
\usepackage[T1]{fontenc}
\usepackage{amsthm}
\usepackage{amsfonts}
\usepackage{txfonts}
\usepackage{palatino,amssymb,epsfig}
\usepackage{latexsym,epsf,epic,amscd}
\usepackage{mathrsfs}
\usepackage{graphicx}
\usepackage{caption}
\usepackage{subcaption}
\usepackage{appendix}
\usepackage{amsmath}
\usepackage{bookmark}
\hypersetup{
colorlinks=true,
linkcolor=black,
citecolor=black,
anchorcolor=black,
urlcolor=black}
\allowdisplaybreaks[4]
\numberwithin{equation}{section}
\DeclareMathOperator{\Hessian}{Hess}

\DeclareMathOperator{\interior}{int}
    \newcommand{\Addresses}{{
  \bigskip
  \footnotesize
  \noindent Te Ba, \href{batexu@hnu.edu.cn}{batexu@hnu.edu.cn}
  \newline\textit{School of Mathematics, Hunan University, Changsha 410082, P.R. China}\par\nopagebreak
  \medskip
  \noindent Shengyu Li, \href{lishengyu@hnu.edu.cn}{lishengyu@hnu.edu.cn}
  \newline\textit{School of Mathematics, Hunan University, Changsha 410082, P.R. China}\par\nopagebreak
  \medskip
  \noindent Yaping Xu, \href{xuyaping@hnu.edu.cn}{xuyaping@hnu.edu.cn}
  \newline\textit{School of Mathematics, Hunan University, Changsha 410082, P.R. China}\par\nopagebreak
}}
\title{Rigidity of bordered polyhedral surfaces}
\author{Te Ba\and Shengyu Li\and Yaping Xu}
\date{}
\newtheorem{theorem}{Theorem}[section]
\newtheorem{lemma}[theorem]{Lemma}
\newtheorem{proposition}[theorem]{Proposition}
\newtheorem{corollary}[theorem]{Corollary}
\theoremstyle{definition}
\newtheorem{definition}[theorem]{Definition}
\newtheorem{remark}[theorem]{Remark}
\begin{document}
\maketitle

\begin{abstract}
This paper investigates the rigidity of bordered polyhedral surfaces. Using the variational principle, we show that bordered polyhedral surfaces are determined by boundary value and discrete curvatures on the interior edges. As a corollary, we reprove the classical result that two Euclidean cyclic polygons (or hyperbolic cyclic polygons) are congruent if the lengths of their sides are equal.

\end{abstract}

\section{Introduction}
\subsection{Background}
A Euclidean (or spherical or hyperbolic) polyhedral surface is a triangulated surface $(S,T)$ with a metric, called a polyhedral metric, so that each triangle in the triangulation is isometric to a Euclidean (or spherical or hyperbolic) triangle. Let $E$, $V$ be the set of edges and vertices of $(S,T)$, the polyhedral metric can be identified as a function $d:E\to\mathbb{R}_{+}$ satisfying the triangle inequality.
The rigidity problem asks which geometric data on $(S,T)$ can determine $d$ up to isometry. Over the past decades, there are many results related to this topic. See e.g. \cite{Luo3,Luo1,Luo2,Guo,Glick,GT,Gu1,Gu2,Xu,Xu3}.
Let $\mathbb{E}^{2}$, $\mathbb{H}^{2}$, $\mathbb{S}^{2}$ be the Euclidean, the hyperbolic, and the spherical $2$-dimensional geometries, respectively. Luo \cite{Luo3,Luo1} introduced the following two families of discrete curvatures.
\begin{definition}\label{def1}
Let $(S,T)$ be a closed triangulated surface. Given a $\mathrm{K}^{2}$ polyhedral metric $d$ on $(S,T)$ where $\mathrm{K}^{2}=\mathbb{E}^{2}$, or $\mathbb{H}^{2}$, or $\mathbb{S}^{2}$, for any $h\in\mathbb{R}$, the $\phi_{h}$-curvature of $d$ is defined by a function $\phi_{h}: E\rightarrow\mathbb{R}$ sending an edge $e$ to
\[\phi_{h}(e)=\int_{a}^{\pi/2}\sin^{h}t\textit{d}t+\int_{a'}^{\pi/2}\sin^{h}t\textit{d}t.\]
The $\psi_{h}$-curvature of $d$ is defined by a function $\psi_{h}: E\rightarrow\mathbb{R}$ sending an edge $e$ to
\[\psi_{h}(e)=\int_{0}^{(b+c-a)/2}\cos^{h}t\textit{d}t+\int_{0}^{(b'+c'-a')/2}\cos^{h}t\textit{d}t,\]
where $a,a'$ are the angles facing the edge $e$ and $b,c,b',c'$ are inner angles adjacent to the edge $e$.
\end{definition}

\begin{figure}[htbp]
\centering

\includegraphics[scale=0.3]{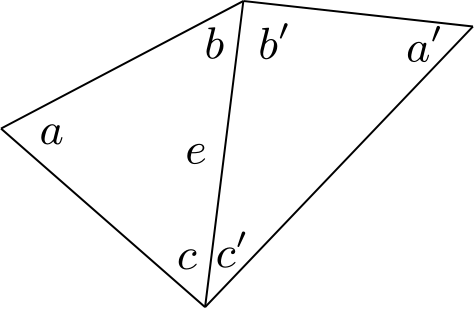}
\caption{\small Edge-angle relation in Definition \ref{def1}.}
\label{fig1}

\end{figure}

Note that some choices for the parameter $h$ are interesting. For example, when $h=0$, one reads
\[\phi_{0}(e)=\pi-a-a'\quad\text{and}\quad\psi_{0}(e)=\frac{1}{2}(b+b'+c+c'-a-a').\]
The geometric meaning of them is related to the dihedral angle along edges of a hyperbolic polyhedron associated to the polyhedral surface. In addition, the curvature
\[\phi_{-2}(e)=\cot a+\cot a'\]
is closely related to the discrete (cotangent) Laplacian operator on a polyhedral surface in Euclidean background geometry. See \cite{Pin,Bo2,Zeng} for more information.
Moreover, a Euclidean or hyperbolic polyhedral surface is Delaunay if and only if $\psi_{h}(e)\geq0$ for all $e\in E$.
\begin{theorem}\cite[Theorem 1.2]{Luo3}\label{th1.1}
Let $(S,T)$ be a closed triangulated surface. For any $h\in\mathbb{R}$, the following statements hold:
\begin{itemize}
\item[$(a)$] A Euclidean polyhedral metric on $(S,T)$ is determined up to scaling by its $\phi_{h}$-curvature or $\psi_{h}$-curvature.
\item[$(b)$] A hyperbolic polyhedral metric on $(S,T)$ is determined by its $\psi_{h}$-curvature.
\item[$(c)$] A spherical polyhedral metric on $(S,T)$ is determined by its $\phi_{h}$-curvature.
\end{itemize}
\end{theorem}
\begin{remark}
Rivin \cite{Ri1} and Leibon \cite{Le} proved Theorem \ref{th1.1}$(a)$, Theorem \ref{th1.1}$(b)$ for $h=0$, respectively. In Euclidean geometry, it is easy to see $a+b+c=\pi$, $a'+b'+c'=\pi$, which implies $\phi_{h}(e)=\psi_{h}(e)$.
\end{remark}
A circle packing metric on $(S,T)$ is a polyhedral metric $d:E\to\mathbb{R}_{+}$
so that there is a map $r:V\to\mathbb{R}_{+}$ satisfying $d(uv)=r(u)+r(v)$ whenever the edge $uv$ has end vertices $u$ and $v$. The $k_{h}$-curvature of the circle packing metric was introduced by Luo \cite{Luoarxiv,Luo1}.
\begin{definition}
Let $(S,T)$ be a closed triangulated surface. Given a circle packing metric $r$ on $(S,T)$, for any $h\in\mathbb{R}$, the $k_{h}$-curvature of $r$ is defined by the function $k_{h}:V\rightarrow\mathbb{R}$ sending a vertex $v$ to
\[k_{h}(v)=(2-\frac{m}{2})\pi-\sum_{i=1}^{m}\int_{\pi/2}^{\theta_{i}}\tan^{h}(t/2)\textit{d}t,\]
where $\theta_{1},\cdots,\theta_{m}$ are all inner angles at $v$.
\end{definition}
Notice that $k_{0}(v)=2\pi-\sum_{i=1}^{m}\theta_{i}$ is the classical discrete curvature. If $e$ is an edge having $v$ as a vertex, we denote it by $v\prec e$. It is easy to show $\sum\nolimits_{v\prec e}\psi_{0}(e)=2\pi-k_{0}(v)$ in Euclidean, hyperbolic and spherical background geometry.
\begin{theorem}\cite[Theorem 1.7]{Luo1}\label{th1.2}
Let $(S,T)$ be a closed triangulated surface. For any $h\in\mathbb{R}$, the following statements hold:
\begin{itemize}
\item[$(a)$] A Euclidean circle packing metric on $(S,T)$ is determined up to scaling by its $k_{h}$-curvature.
\item[$(b)$] A hyperbolic circle packing metric on $(S,T)$ is determined by its $k_{h}$-curvature.
\end{itemize}
\end{theorem}
\begin{remark}
Theorem \ref{th1.2}$(a)$ for $h=0$ was proved by Andreev \cite{An} and Thurston \cite{Thu}. Theorem \ref{th1.2}$(b)$ for $h=0$ was proved by Thurston \cite{Thu}.
\end{remark}
\subsection{Main results}
The polyhedral surfaces in above works are required to be closed. One may ask whether analogous results still hold for bordered polyhedral surfaces. Actually, there are two questions.
The first one is parallelling to this for closed case: Whether a polyhedral metric is determined by its discrete curvature. This was answered by the work of Ba-Zhou \cite{Ba}, which shows a similar result to Theorem \ref{th1.1}.
The second question asks whether a polyhedral metric is determined by its discrete curvature on interior edges together with its boundary value. We mention that this problem plays a significant role in exploring rigidity of cyclic polygons. See Corollary \ref{co} for details.

For an edge $e$ of a bordered polyhedral surface $(S,T)$, we call $e$ a boundary edge if $e$ belongs to the boundary of $S$. Otherwise, we call $e$ an interior edge. Similarly, we can define boundary vertices and interior vertices. For a polyhedral metric $d:E\to\mathbb{R}_{+}$ on $(S,T)$, the boundary value of $d$ is the restriction of $d$ to $E_{\partial}$, where $E_{\partial}$ is the set of boundary edges. In this paper we prove the following result.

\begin{theorem}\label{mt1}
Let $(S,T)$ be a bordered triangulated surface. For any $h\in\mathbb{R}$, the following statements hold:
\begin{itemize}
\item[$(a)$] A Euclidean polyhedral metric on $(S,T)$ is determined up to isometry by its boundary value and $\phi_{h}$-curvature (or $\psi_{h}$-curvature) on interior edges.
\item[$(b)$] A hyperbolic polyhedral metric on $(S,T)$ is determined up to isometry by its boundary value and $\psi_{h}$-curvature on interior edges.
\item[$(c)$] A spherical polyhedral metric on $(S,T)$ is determined up to isometry by its boundary value and $\phi_{h}$-curvature on interior edges.
\end{itemize}
\end{theorem}
A Euclidean (or spherical) polygon is said to be cyclic if it can be inscribed into a Euclidean (or spherical) circle.
A hyperbolic polygon is cyclic if it can be inscribed into a hyperbolic circle, or a hypercycle, or a horocycle.  Let $P$ be a cyclic polygon. Then $P$ can be realized as a bordered polyhedral surface by adding edges between appropriate pairs of vertices of $P$. We denote it by $(P,T)$. See Figure \ref{fig2} for an example. Notice that the boundary value of $(P,T)$ is the side lengths of $P$. Lemma \ref{cyc} indicates that $\psi_{0}$-curvature on each interior edge of $(P,T)$ equals to zero. As a corollary of Theorem \ref{mt1}, we immediately obtain the rigidity of Euclidean and hyperbolic cyclic polygon. Refer e.g. \cite{Pe,Pi,Sc,Ko} for other proofs. The rigidity of spherical cyclic polygon is the straightforward corollary of the rigidity of Euclidean cyclic polygon. See \cite{Ko} for details.
\begin{figure}[htbp]
\centering

\includegraphics[scale=0.27]{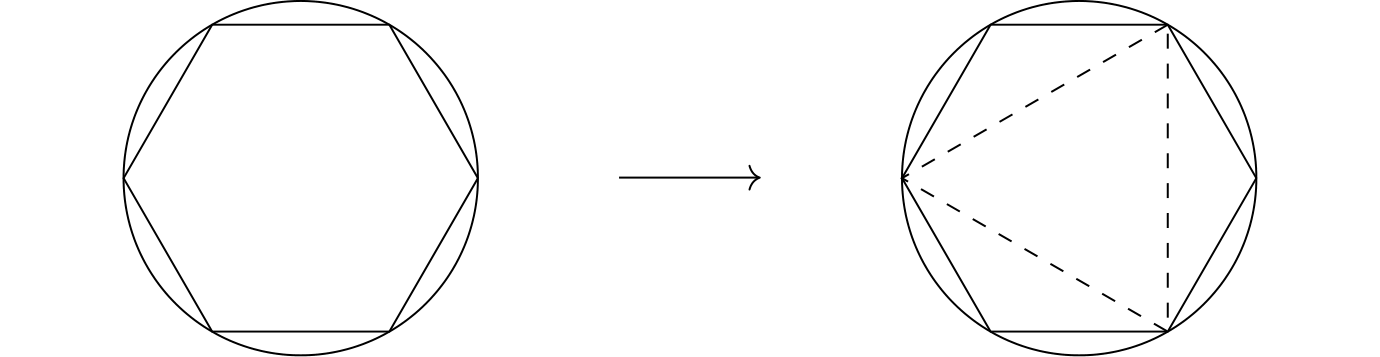}
\caption{\small Euclidean triangulated cyclic hexagon}
\label{fig2}

\end{figure}

\begin{corollary}\label{co}
A Euclidean (or hyperbolic) cyclic polygon is determined by its lengths of sides.
\end{corollary}
In \cite{Luo3,Luo1}, Luo asked whether $\phi_{h}$-curvature determines a hyperbolic polyhedral metric. In the following we give a partial answer to this question. We say a triangulated surface $(S,T)$ is stripped if every triangle of $T$ has at least one boundary edge. For example, the left (resp. right) in Figure \ref{fig3} is a stripped (resp. non-stripped) triangulated surface.

\begin{figure}[htbp]
\centering

\includegraphics[scale=0.27]{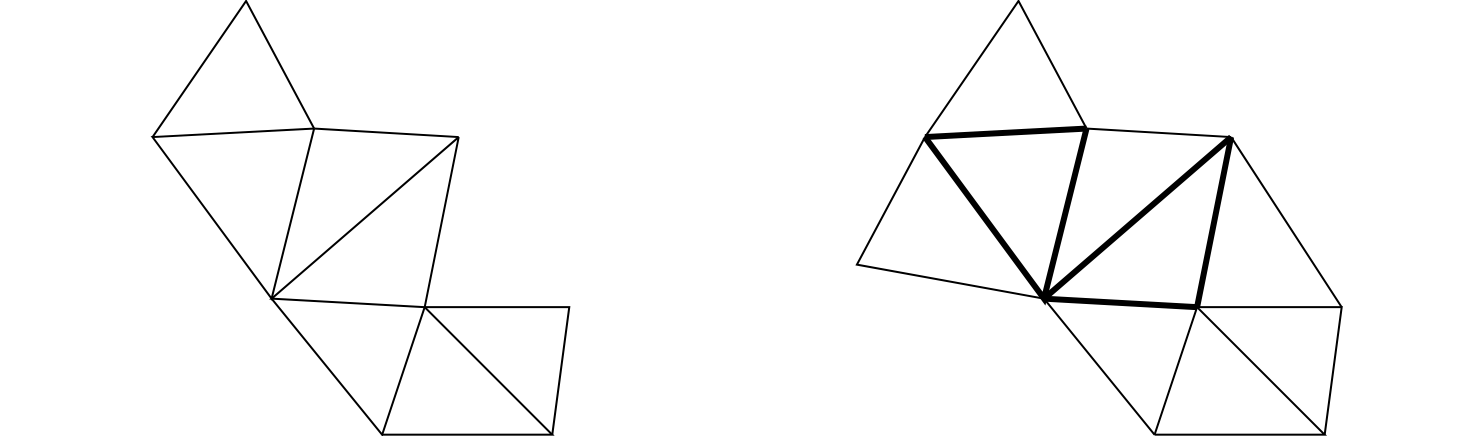}
\caption{\small Stripped and non-stripped triangulated surfaces}
\label{fig3}

\end{figure}

\begin{theorem}\label{1.7}
Let $(S,T)$ be a stripped triangulated surface. For any $h\in\mathbb{R}$, a hyperbolic polyhedral metric on $(S,T)$ is determined up to isometry by its boundary value and $\phi_{h}$-curvature of interior edges.
\end{theorem}
For a circle packing metric $r:V\to\mathbb{R}_{+}$, the boundary value is the restriction of $r$ to $V_{\partial}$, where $V_{\partial}$ is the set of boundary vertices.
\begin{theorem}\label{mt2}
Let $(S,T)$ be a bordered triangulated surface. For any $h\in\mathbb{R}$, the following statements hold:
\begin{itemize}
\item[$(a)$] A Euclidean circle packing metric on $(S,T)$ is determined by its boundary values and $k_{h}$-curvature on interior vertices.
\item[$(b)$] A hyperbolic circle packing metric on $(S,T)$ is determined by its boundary values and $k_{h}$-curvature on interior vertices.
\end{itemize}
\end{theorem}

Theorem \ref{mt2} can be applied to study the rigidity of planar circle packing. A planar circle packing is a connected collection of circles on Euclidean plane whose interiors are disjoint. The contact graph of a circle packing is the graph having a vertex for each circle, and an edge for every pair of circles that are tangent. If there exists a bordered triangulated surface whose nerve is isomorphic to the contact graph of the circle packing, the circle packing is called a triangle-type circle packing. Suppose $P$ is a triangle-type circle packing. If a circle is corresponding to a boundary vertex of the bordered triangulated surface under the isomorphism, the circle is called a boundary circle in $P$. As a corollary Theorem \ref{mt2}, we obtain the following well-known result regarding the rigidity of planar circle packing. See \cite{Ste,Con} for more discussions.

\begin{figure}[htbp]
\centering

\includegraphics[scale=0.27]{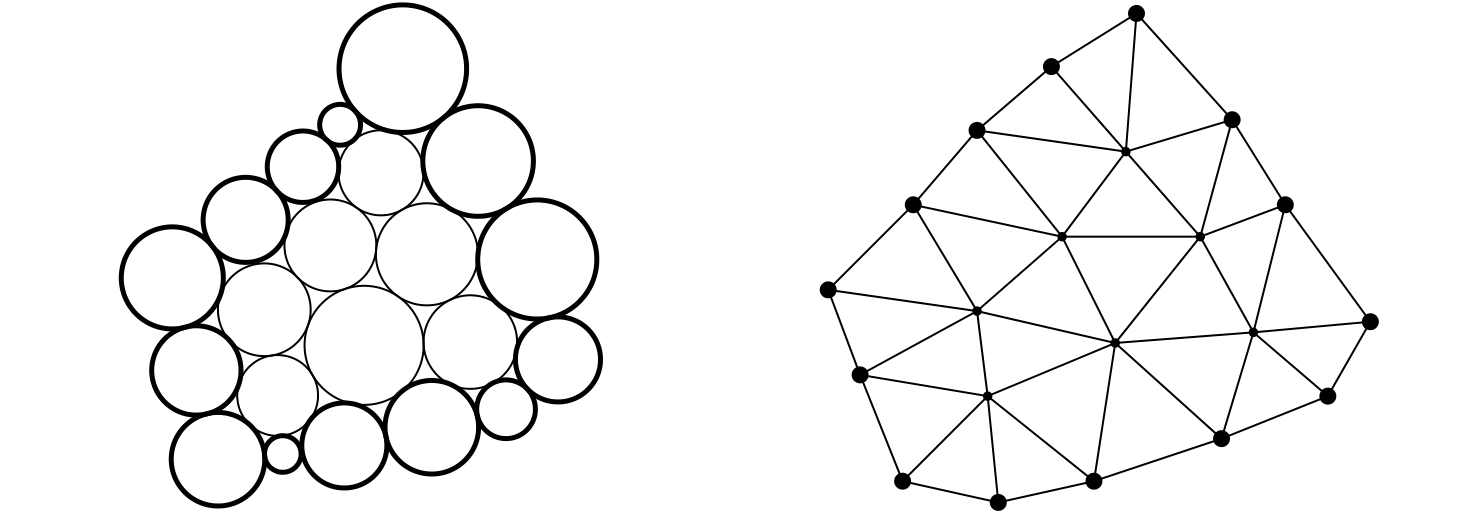}
\caption{\small A circle packing and its contact graph}
\label{fig4}

\end{figure}

\begin{corollary}\label{cocp}
A triangle-type planar circle packing on the plane is determined up to isometry by the size of boundary circles in the sense of keeping the contact graph unchanged.
\end{corollary}
The variational principle is applied to proofs of the above theorems. The use of variational principle in the study of triangulated surfaces first appeared in the work of Colin de Verdi\`{e}re \cite{Co} to give a new proof Koebe-Andreev-Thurston's circle packing theorem. Similar methods have been used by Rivin \cite{Ri1,Ri2}, Leibon \cite{Le}, Bobenko-Springborn \cite{Bo1} in other situations. We are also highly influenced by the works of Luo \cite{Luoarxiv,Luo3,Luo1}, Luo-Dai-Gu~\cite{L.D.G}, Guo-Luo \cite{Luo2}, Guo \cite{Guo}, Zhou \cite{Zhou1,Zhou2}, Xu \cite{Xu,Xu2}, Ge-Hua-Zhou \cite{Ge1,Ge2} and others.

The present paper is built up as follows: In Section \ref{2}, we establish some preliminary results. In Section \ref{3} and \ref{4}, we construct several convex energy functionals on polyhedral surfaces. In Section \ref{5}, applying variational principle, we derive Theorem \ref{mt1}, Corollary \ref{co}, Theorem \ref{1.7} and Theorem \ref{mt2}. The last two sections contain appendixes regarding to some results about trigonometry and differential forms. Throughout this paper, we always assume the polyhedral surface is connected unless otherwise stated.
\section{Preliminaries}\label{2}
\subsection{Basic results about cosine law}\label{section2.1}
We first introduce some basic results regarding to the cosine law function.
Given a geometric triangle, the cosine law expresses the inner angles in terms of three edge lengths. See Appendix \ref{AppA} for details. So we regard the inner angles of the geometric triangle as the function of three edge lengths, called cosine law function. The following three lemmas establish some basic properties of cosine law function. Please refer to \cite{chow,L.D.G,Zhou1,Ge1,Ge2} for more information.
\begin{lemma}\label{2.1}
Let $l_{s}$, $\theta_{s}$ ($s=1,2,3$) be the edge length and opposite inner angle of a geometric triangle such that $\theta_{s}$ facing $l_{s}$. Set $\{i,j,k\}=\{1,2,3\}$. Then
\[
\frac{\partial\theta_{i}}{\partial l_{i}}=\frac{m_{i}}{\sin\theta_{i}m_{j}m_{k}},\qquad\frac{\partial\theta_{i}}{\partial l_{j}}=-\cos\theta_{k}\frac{\partial\theta_{i}}{\partial l_{i}},
\]
where $m_{i}=l_{i}$ (resp. $m_{i}=\sinh l_{i}$, $m_{i}=\sin l_{i}$) in Euclidean (resp. hyperbolic, spherical) background geometry.
\end{lemma}
\begin{lemma}\cite[Lemma 2.3]{chow}\label{silemma}
Let $l_{s}$, $\theta_{s}$ ($s=1,2,3$) be the edge length and opposite inner angle of a geometric triangle such that $\theta_{s}$ facing $l_{s}$. Set $\{i,j,k\}=\{1,2,3\}$ and $r_{i}=\frac{1}{2}(l_{j}+l_{k}-l_{i})$. Then
\[
\frac{1}{s_{i}}\frac{\partial\theta_{i}}{\partial r_{j}}=\frac{1}{s_{j}}\frac{\partial\theta_{j}}{\partial r_{i}},
\]
where $s_{j}=r_{j}$ (resp. $s_{j}=\sinh r_{j}$) in Euclidean (resp. hyperbolic) background geometry.
\end{lemma}

\begin{lemma}\cite[Lemma 2.2]{chow}\label{chow-luo}
Let $l_{s}$, $\theta_{s}$ ($s=1,2,3$) be the edge length and opposite inner angle of a geometric triangle such that $\theta_{s}$ facing $l_{s}$. Set $\{i,j,k\}=\{1,2,3\}$ and $r_{i}=\frac{1}{2}(l_{j}+l_{k}-l_{i})$. Then
\[\frac{\partial(\theta_{1}+\theta_{2}+\theta_{3})}{\partial r_{i}}<0\]
in hyperbolic background geometry.
\end{lemma}
Based on Lemma \ref{2.1}, we further have the following result.
\begin{lemma}\label{2.3}
Let $l_{s}$, $\theta_{s}$ ($s=1,2,3$) be the edge length and opposite inner angle of a hyperbolic triangle such that $\theta_{s}$ facing $l_{s}$. Set $\{i,j,k\}=\{1,2,3\}$. Then

\begin{equation}
\label{a}
\frac{\partial(\theta_{i}-\theta_{j}-\theta_{k})}{\partial l_{i}}=\frac{\tanh l_{i}/2}{\sin\theta_{i}\sinh l_{j}\sinh l_{k}}(\cosh l_{1}+\cosh l_{2}+\cosh l_{3}+1),
\end{equation}
\begin{equation}
\label{b}
\frac{\partial(\theta_{i}-\theta_{j}-\theta_{k})}{\partial l_{j}}=\frac{\coth l_{j}/2}{\sin\theta_{i}\sinh l_{j}\sinh l_{k}}(-\cosh l_{i}-\cosh l_{j}+\cosh l_{k}+1).
\end{equation}

\end{lemma}
\begin{proof}
By the sine law, we know that
\[c_{ijk}=\sin\theta_{i}\sinh l_{j}\sinh l_{k}\]
is independent of indices.
From Lemma \ref{2.1}, we derive
\begin{equation}\label{p1}\begin{aligned}
\frac{\partial(\theta_{i}-\theta_{j}-\theta_{k})}{\partial l_{i}}=\frac{1}{c}\left(\sinh l_{i}+\cos\theta_{k}\sinh l_{j}+\cos\theta_{j}\sinh l_{k}\right),
\end{aligned}\end{equation}
where $c=\sin\theta_{i}\sinh l_{j}\sinh l_{k}$.
Substituting the expression of $\cos\theta_{j}$, $\cos\theta_{k}$ to (\ref{p1}), we have
\[\begin{aligned}
\frac{\partial(\theta_{i}-\theta_{j}-\theta_{k})}{\partial l_{i}}&=\frac{1}{c \sinh l_{i}}\left(-\cosh l_{k}+\cosh l_{i}\cosh l_{j}-\cosh l_{j}+\cosh l_{i}\cosh l_{k}+\cosh^{2}l_{i}-1\right)\\
&=\frac{1}{c \sinh l_{i}}\Big((\cosh l_{i}-1)(\cosh l_{j}+\cosh l_{k})+(\cosh l_{i}-1)(\cosh l_{i}+1)\Big)\\
&=\frac{\tanh\frac{l_{i}}{2}}{c}(\cosh l_{i}+\cosh l_{j}+\cosh l_{k}+1).
\end{aligned}\]
This proves (\ref{a}). Similar computation yields (\ref{b}) as well.
\end{proof}
The following two lemmas play a crucial role in the proof of the convexity of energy functionals constructed in Section \ref{3}.
\begin{lemma}\label{hypermatrix}
Let $l_{1}$, $l_{2}$, $l_{3}$ be the edge lengths of a hyperbolic triangle. Set $\{i,j,k\}=\{1,2,3\}$.
Define $\mathrm{P}=[p_{st}]$, where
\[
p_{ii}=\tanh^{2}\frac{l_{i}}{2}(\cosh l_{1}+\cosh l_{2}+\cosh l_{3}+1),\quad
p_{ij}=-\cosh l_{i}-\cosh l_{j}+\cosh l_{k}+1.
\]
Then $\mathrm{P}$ is positive definite.
\end{lemma}
\begin{proof}
For the simplicity of the proof, some substitutions to $\mathrm{P}$ are needed. By some basic trigonometry identities and Proposition \ref{cosine}, we have
\[\begin{aligned}
\tanh^{2}{\frac{l_{i}}{2}}&=\frac{\cosh l_{i}-1}{\cosh l_{i}+1}\\
&=\frac{\cos\theta_{i}+\cos\theta_{j}\cos\theta_{k}-\sin\theta_{j}\sin\theta_{k}}{\cos\theta_{i}+\cos\theta_{j}\cos\theta_{k}+\sin\theta_{j}\sin\theta_{k}}\\
&=\frac{\cos(\theta_{j}+\theta_{k})+\cos\theta_{i}}{\cos(\theta_{j}-\theta_{k})+\cos\theta_{i}}\\
&=\frac{\cos\left(\frac{\theta_{i}+\theta_{j}+\theta_{k}}{2}\right)\cos\left(\frac{\theta_{j}+\theta_{k}-\theta_{i}}{2}\right)}{\cos\left(\frac{\theta_{i}+\theta_{j}-\theta_{k}}{2}\right)\cos\left(\frac{\theta_{i}+\theta_{k}-\theta_{j}}{2}\right)},
\end{aligned}\]
where $\theta_{i}$ is the inner angle facing $l_{i}$. Set $\zeta=-\frac{1}{2}(\theta_{1}+\theta_{2}+\theta_{3})$ and $\alpha_{i}=\frac{1}{2}(\theta_{i}-\theta_{j}-\theta_{k})$. It follows that
\begin{equation}\label{tanh}\tanh^{2}{\frac{l_{i}}{2}}=\frac{\cos\zeta\cos\alpha_{i}}{\cos\alpha_{j}\cos\alpha_{k}}.\end{equation}
By Proposition \ref{cosine}, we also derive
\begin{equation}\label{c123}\begin{aligned}\cosh l_{1}+\cosh l_{2}+\cosh l_{3}+1
&=\frac{\big(\sin(\theta_{2}+\theta_{3})+\sin\theta_{1}\big)\big(\cos(\theta_{2}-\theta_{3})+\cos\theta_{1}\big)}{\sin\theta_{1}\sin\theta_{2}\sin\theta_{3}}\\
&=\frac{-4\sin\zeta\cos\alpha_{1}\cos\alpha_{2}\cos\alpha_{3}}{\sin\theta_{1}\sin\theta_{2}\sin\theta_{3}}.
\end{aligned}\end{equation}
Combining (\ref{tanh}) and (\ref{c123}), we derive
\begin{equation}\label{tanh1}\begin{aligned}
\tanh^{2}{\frac{l_{i}}{2}}(\cosh l_{1}+\cosh l_{2}+\cosh l_{3}+1)
=\frac{-4\cos\zeta\sin\zeta\cos^{2}\alpha_{i}}{\sin\theta_{1}\sin\theta_{2}\sin\theta_{3}}.
\end{aligned}\end{equation}
Similar computation to (\ref{c123}) yields
\begin{equation}\label{cij}
-\cosh l_{i}-\cosh l_{j}+\cosh l_{k}+1=\frac{4\cos\zeta\cos \alpha_{i}\cos \alpha_{j}\sin \alpha_{k}}{\sin\theta_{1}\sin\theta_{2}\sin\theta_{3}}.
\end{equation}
Set $\mathrm{D}=[d_{st}]$, where $d_{ii}=\cos\alpha_{i}$, $d_{ij}=0$ and $\mathrm{Q}=[q_{st}]$, where $q_{ii}=-\sin\zeta$, $q_{ij}=\sin\alpha_{k}$.
Combining (\ref{tanh1}) and (\ref{cij}), we have
\[
\mathrm{P}=\frac{4\cos\zeta}{\sin\theta_{1}\sin\theta_{2}\sin\theta_{3}}\mathrm{D}\mathrm{Q}\mathrm{D}.
\]
Now it remains to prove the positive-definiteness of $\mathrm{Q}$. We will prove it by showing that determinant of sequential principal minor of $\mathrm{Q}$ is positive. Obviously, we have $-\sin\zeta>0$.
The determinant of $2$-principal minor of $\mathrm{Q}$ equals to
\[\begin{aligned}
&(\sin\zeta+\sin\alpha_{3})(\sin\zeta-\sin \alpha_{3})\\
=&\ 4\sin\frac{\zeta+\alpha_{3}}{2}\cos\frac{\alpha_{1}+\alpha_{2}}{2}\cos\frac{\zeta+\alpha_{3}}{2}\sin\frac{\alpha_{1}+\alpha_{2}}{2}\\
=&\sin(\zeta+\alpha_{3})\sin(\alpha_{1}+\alpha_{2})\\
=&\sin(\theta_{1}+\theta_{2})\sin\theta_{3}>0.
\end{aligned}\]
Some tedious computation yields
\[
\begin{aligned}
\det\mathrm{Q}&=\sin\zeta\cos^{2}\zeta-\frac{1}{2}\cos\zeta\sum\nolimits_{i=1}^{3}\sin2\alpha_{i}\\
&=\frac{1}{2}\cos\zeta\left(\sin2\zeta-\sum\nolimits_{i=1}^{3}\sin2\alpha_{i}\right)\\
&=\cos\zeta\big(\cos(\theta_{1}-\theta_{2})-\cos(\theta_{1}+\theta_{2})\big)\sin\theta_{3}\\
&=4\cos\left(\frac{\theta_{1}+\theta_{2}+\theta_{3}}{2}\right)\sin\theta_{1}\sin\theta_{2}\sin\theta_{3}>0.
\end{aligned}
\]
We thus finish the proof.
\end{proof}
\begin{lemma}\label{A}
Let $\theta_{1}$, $\theta_{2}$, $\theta_{3}$ be inner angles of a geometric triangle. Set $\{i,j,k\}=\{1,2,3\}$. Set $\mathrm{M}=[m_{st}]$, where $p_{ii}=1$, $p_{ij}=-\cos\theta_{k}$. Then $\mathrm{M}$ is semi-positive definite (resp. positive definite) in Euclidean (resp. spherical) geometry.
\end{lemma}
\begin{remark}
Lemma \ref{hypermatrix} was also proved in \cite{Luoarxiv,L.D.G,Luo1}. Comparing with the proof of in \cite{Luoarxiv,L.D.G,Luo1}, our proof is purely computational.
\end{remark}
\subsection{Differential forms on the moduli spaces of triangles}
In this subsection, we introduce several closed differential forms defined on the moduli space of geometric triangles. One can refer to Appendix \ref{AppB} for some basic results about differential forms. The moduli space of a geometric triangle under polyhedral metric can be written as
\[\begin{aligned}
\Omega^{E}_{3}=\Omega^{H}_{3}&=\left\{(l_{1},l_{2},l_{3})\in\mathbb{R}^{3}_{+}\ \big|\ l_{i}+l_{j}>l_{k},\{i,j,k\}=\{1,2,3\}\right\},\\
\Omega^{S}_{3}&=\left\{(l_{1},l_{2},l_{3})\in(0,\pi)^{3}\ \big|\ l_{i}+l_{j}>l_{k},\ l_{1}+l_{2}+l_{3}<2\pi\right\},
\end{aligned}\]
in Euclidean, hyperbolic and spherical geometry, respectively. If we restrict the length of one edge, such as $l_{3}$ is a positive constant, then the moduli space of the triangle is a subset of $\mathbb{R}^2_{+}$, denote as $\Omega^{E}_{2}$, $\Omega^{H}_{2}$ and $\Omega^{S}_{2}$, respectively. Similarly, if $l_{2}, l_{3}$ are constants, we write the moduli space as $\Omega^{E}_{1}$, $\Omega^{H}_{1}$ and $\Omega^{S}_{1}$, respectively.
Specifically, if the polyhedral metric is a circle packing metric, the moduli space of triangles can be regard as the admissible space of the radius assignment, which is $\mathbb{R}^3_{+}$. The moduli space of triangles which fix the radius assignments of one, (resp. two) vertices is $\mathbb{R}^2_{+}$ (resp. $\mathbb{R}_{+}$). Set
\begin{equation}\label{forms}\begin{aligned}
             &\omega^{E}_{\phi,s}:=\sum_{i=1}^{s}\frac{\int_{\pi/2}^{\theta_{i}}\sin^{h}{t}\mathrm{d}t}{l_{i}^{h+1}}\mathrm{d}l_{i},\
             \omega^{H}_{\psi,s}:=\sum_{i=1}^{s}\frac{\int_{0}^{\alpha_{i}}\cos^{h}{t}\mathrm{d}t}{\tanh^{h+1}{l_{i}/2}}\mathrm{d}l_{i},\
             \omega^{S}_{\phi,s}:=\sum_{i=1}^{s}\frac{\int_{\pi/2}^{\theta_{i}}\sin^{h}{t}\mathrm{d}t}{\sin^{h+1}{l_{i}}}\mathrm{d}l_{i},\\
             &\omega^{H}_{\phi,s}:=\sum_{i=1}^{s}\frac{\int_{\pi/2}^{\theta_{i}}\sinh^{h}{t}\mathrm{d}t}{\sinh^{h+1}l_{i}}\mathrm{d}l_{i},\
             \eta^{E}_{s}:=\sum_{i=1}^{s}\frac{\int_{\pi/2}^{\theta_{i}}\tan^{h}{\frac{t}{2}}\mathrm{d}t}{r_{i}^{-h+1}}\mathrm{d}r_{i},\ \
             \eta^{H}_{s}:=\sum_{i=1}^{s}\frac{\int_{0}^{\theta_{i}}\tan^{h}\frac{t}{2}\mathrm{d}t}{\sinh^{-h+1}r_{i}}\mathrm{d}r_{i},
\end{aligned}\end{equation}
where $\alpha_{i}=\frac{1}{2}(\theta_{i}-\theta_{j}-\theta_{k})$ and $r_{i}=\frac{1}{2}(l_{j}+l_{k}-l_{i})$.
The first four differential forms are defined on the moduli space of triangles for polyhedral metric. The last two differential forms are defined on the moduli space of triangles for circle packing metric. Using Lemma \ref{2.1}, Lemma \ref{silemma}, Lemma \ref{2.3} and proposition \ref{tan1}, the following result can be verified easily.
\begin{lemma}\label{path}
The differential forms defined by (\ref{forms}) are closed.
\end{lemma}
\begin{remark}
When $s=3$, the differential forms $\omega^{E}_{\phi,s}$, $\omega^{H}_{\psi,s}$, $\omega^{S}_{\phi,s}$, $\eta^{E}_{s}$ and $\eta^{H}_{s}$ first appeared in the work of Luo \cite{Luoarxiv}, which also proved the closeness of these differential forms. For the need to prove our main results, we consider above differential forms when $s=1,2$ and define $\omega^{H}_{\phi,s}$.
\end{remark}

\section{Energy functionals on the moduli spaces of polyhedral metrics}\label{3}
\subsection{Energy functionals on the moduli spaces of triangles}\label{Section3.1}
We start from the energy functional defined on the moduli spaces of triangles. According to the number of interior edges of the triangle of polyhedral surface, we divide the triangles into three types.
We begin by introducing a continuous map defined on the moduli space of triangles. Let $\tau$ be a Euclidean (or hyperbolic or spherical) triangle in a bordered triangulated surface and let $l=(l_{1},l_{2},l_{3})\in\mathbb{R}^{3}_{+}$ represent the edge length of $\tau$. Then the continuous map $\xi$ is defined as
\begin{equation}\label{xi}
\xi(l)=\big(\xi_{h}(l_{1}),\xi_{h}(l_{2}),\xi_{h}(l_{3})\big),
\end{equation}
where
\[\xi_{h}(t)=\left\{
             \begin{aligned}
             &\frac{-t^{-h}}{h},\ h\neq0,&\tau\ \text{is a Euclidean triangle}, \\
             &\ln{t},\quad h=0,&\tau\ \text{is a Euclidean triangle},\\
             &\int^{t}_{1}\tanh^{-h-1}\frac{x}{2}\mathrm{d}x,&\tau\ \text{is a hyperbolic triangle},\\
             &\int^{t}_{1}\sin^{-h-1}{x}\mathrm{d}x.&\tau\ \text{is a spherical triangle}.
             \end{aligned}
\right.\]
When $\tau$ is a hyperbolic triangle, another continuous map $\gamma$ is defined as
\begin{equation}\label{gamma}\gamma(l)=\big(\gamma_{h}(l_{1}),\gamma_{h}(l_{2}),\gamma_{h}(l_{3})\big),\end{equation}
where
\[\gamma_{h}(t)=\int^{t}_{1}\sinh^{-h-1}{x}\mathrm{d}x.\]
If we consider the moduli space of triangles with one or two interior edges, similar map can be defined on $\Omega^{E}_{s}$, $\Omega^{H}_{s}$, $\Omega^{S}_{s}$ when $s=1,2$. We still denote them by $\xi$ and $\gamma$. Then the following result is straightforward because of the monotonicity of $\xi_{h}(t)$, $\gamma_{h}(t)$.
\begin{lemma}\label{homeomorphism}
The continuous maps $\xi$, $\gamma$ defined in (\ref{xi}),(\ref{gamma}) are homeomorphic to its image.
\end{lemma}
Now we introduce several energy functionals on the moduli space of triangles. For $s=1,2,3$, we define line integral function
\begin{equation}\label{function}
             \begin{aligned}
             F^{E}_{s}(u):=\int^{u}_{-h}\omega^{E}_{\phi,s}, &\ u\in \xi(\Omega^{E}_{s}),\quad
             F^{H}_{s}(u):=\int^{u}_{0}\omega^{H}_{\psi,s}, \  u\in\xi(\Omega^{H}_{s}),\\
             F^{S}_{s}(u):=\int^{u}_{0}\omega^{S}_{\phi,s}, &\  u\in\xi(\Omega^{S}_{s}),\quad
             G^{H}_{s}(u):=\int^{u}_{0}\omega^{H}_{\phi,s}, \  u\in\gamma(\Omega^{H}_{s}).
             \end{aligned}
\end{equation}
By Poincar\'{e}'s Lemma and Lemma \ref{path}, above functions are well-defined.
\begin{remark}
Function $F^{E}_{s}(u)$, $F^{H}_{s}(u)$ and $F^{S}_{s}(u)$ were first introduced in \cite{Luoarxiv} when $s=3$.
\end{remark}
\begin{lemma} \label{locconv}
Let $F^{E}_{s}$, $F^{H}_{s}$, $F^{S}_{s}$, $G^{H}_{s}$ be defined in ($\ref{function}$). The following statements hold:
\begin{itemize}
\item[$(a)$]The integral $F^{E}_{s}$ is locally convex when $s=3$ and locally strictly convex when $s=1,2$.
\item[$(b)$]The integral $F^{S}_{s}$ is locally strictly convex for $s=1,2,3$.
\item[$(c)$]The integral $G^{H}_{s}$ is locally strictly convex when $s=1,2$.
\item[$(d)$]The integral $F^{H}_{s}$ is locally strictly convex for $s=1,2,3$.
\end{itemize}
\end{lemma}
\begin{proof}
Some computations indicate that the function defined in ($\ref{function}$) is smooth. It suffices to verify the convexity by computing the Hessian matrix. Let us first consider the Hessian matrix of $F^{E}_{3}(u)$. Observe that
\begin{equation}\label{ad1}            \omega^{E}_{\phi,3}=\sum_{i=1}^{3}\frac{\int_{\pi/2}^{\theta_{i}}\sin^{h}{t}\mathrm{d}t}{l_{i}^{h+1}}\mathrm{d}l_{i}=\sum_{i=1}^{3}\int^{\theta_{i}}_{\pi/2}\sin^{h}t\mathrm{d}t\mathrm{d}u_{i}.
\end{equation}
It follows that
\begin{equation}\label{ad2}
\frac{\partial F^{E}_{3}(u)}{\partial u_{i}}=\int^{\theta_{i}}_{\pi/2}\sin^{h}t\mathrm{d}t.
\end{equation}
Then we obtain
\begin{equation}\label{ad3}\begin{aligned}
\frac{\partial^{2}F^{E}_{3}(u)}{\partial u_{i}\partial u_{j}}=\sin^{h}\theta_{i}\frac{\partial\theta_{i}}{\partial u_{j}}
=\sin^{h}\theta_{i}\frac{\partial\theta_{i}/\partial l_{j}}{\xi'_{h}(l_{j})}.
\end{aligned}
\end{equation}
By Lemma \ref{2.1}, we deduce
\begin{equation}\label{2-derivative}
\frac{\partial^{2}F^{E}_{3}(u)}{\partial u_{i}^{2}}=cl^{2h+2}_{i},\quad\quad
\frac{\partial^{2}F^{E}_{3}(u)}{\partial u_{i}\partial u_{j}}=cl^{h+1}_{i}l^{h+1}_{j}\cos\theta_{k},
\end{equation}
where
\[c=\frac{\sin^{h-1}\theta_{i}}{l^{h}_{i}l_{j}l_{k}}\]
is independent of the index. Set $\mathrm{D_{1}}=[a_{st}]$, where $a_{ij}=0$, $a_{ii}=l_{i}^{h+1}$. Recall that $\mathrm{M}=[m_{st}]$, where $p_{ii}=1$, $p_{ij}=-\cos\theta_{k}$. Then the Hessian matrix of $F^{E}_{3}(u)$ can be represented as $c\mathrm{D_{1}}\mathrm{M}\mathrm{D_{1}}$. We use $\mathrm{A'}$ to represent the submatrix of $\mathrm{A}$, obtained by removing the last row and the last column of matrix $\mathrm{A}$. When $s=2$, similar reasoning yields that the Hessian matrix of $F^{E}_{2}(u)$ can be represented as $c\mathrm{D_{1}'}\mathrm{M'}\mathrm{D'_{1}}$. Now the positive-definiteness of Hessian matrix of $F^{E}_{3}(u)$, $F^{E}_{2}(u)$ is equivalent to the positive-definiteness of $\mathrm{M}$ and $\mathrm{M'}$, which is proved by Lemma \ref{A}. When $s=1$, the strictly convexity of $F^{E}_{1}(u)$ is the direct result of (\ref{2-derivative}). Thus $(a)$ is proved.
By Lemma \ref{2.1} and sine law, some computations show that the second partial derivative of $F^{S}_{3}(u)$ and $G^{H}_{3}(u)$ have the same form as (\ref{2-derivative}). Lemma \ref{A} indicates $\mathrm{M}$ is strictly positive-definite in spherical background geometry. Thus $(b)$ is proved. Similar analysis yields $(c)$ easily.

Next we compute the Hessian matrix of $F^{H}_{3}(u)$. Remember that $\alpha_{i}=\frac{1}{2}(\theta_{i}-\theta_{j}-\theta_{k})$. Similar computation to (\ref{ad1})-(\ref{ad3}) yields that
\[\begin{aligned}
\frac{\partial^{2}F^{H}_{3}(u)}{\partial u_{i}\partial u_{j}}=\cos^{h}\alpha_{i}\frac{\partial \alpha_{i}}{\partial u_{j}}
=\frac{\cos^{h}\alpha_{i}}{\gamma'_{h}(l_{j})}\frac{\partial\alpha_{i}}{\partial l_{j}}.
\end{aligned}\]
By Lemma \ref{2.3} we obtain
\[\begin{aligned}
\frac{\partial^{2}F^{H}_{3}(u)}{\partial u_{i}^{2}}&=m_{2}\tanh^{2h+2}\frac{l_{i}}{2}(\cosh l_{1}+\cosh l_{2}+\cosh l_{3}+1),\\
\frac{\partial^{2}F^{\mathrm{H}}_{3}(u)}{\partial u_{i}\partial u_{j}}&=m_{2}\tanh^{h}\frac{l_{i}}{2}\tanh^{h}\frac{l_{j}}{2}(-\cosh l_{i}-\cosh l_{j}+\cosh l_{k}+1),
\end{aligned}\]
where \[m_{2}=\frac{\cos^{h}\alpha_{i}}{\tanh^{h}\frac{l_{i}}{2}\sin\theta_{i}\sinh l_{j}\sinh l_{k}}.\] Lemma \ref{tan1} and sine law yield that $m_{2}$ is independent of the indices. Set $\mathrm{D_{2}}=[b_{st}]$, where $b_{ij}=0$, $b_{ii}=\tan^{h}\frac{l_{i}}{2}$. Then the Hessian matrix of $F^{H}_{3}(u)$ can be represented as $c\mathrm{D_{2}}\mathrm{P}\mathrm{D_{2}}$, where $\mathrm{P}$ is the matrix defined in Lemma \ref{hypermatrix} and $c$ is a positive number. Thus $(d)$ is the result of Lemma \ref{hypermatrix}.

\end{proof}

\subsection{Extension of locally convex functionals}
We establish some simple facts on extending locally convex energy functionals to convex functionals in this subsection. Suppose $A$ is a subspace of a topological space $X$ and $f:A\rightarrow Y$ is continuous. If there exists a continuous function $F:X\rightarrow Y$ such that $F|_{A}=f$ and $F$ is a constant function on each connected component of $X-A$, then we say $f$ can be extended continuously by constant functions to $X$. In this subsection, we define $J=\mathbb{R}_{+}$ in Euclidean and hyperbolic geometry and $J=(0,\pi)$ in spherical geometry.
\begin{lemma}\cite[Lemma 2.2]{Luo3}\label{ex.cont.3}
Let $l_{i}$, $\theta_{i}$ ($i=1,2,3$) be the edge length and opposite inner angle of a geometric triangle such that $\theta_{i}$ facing $l_{i}$. The cosine law function $\theta_{i}(l_{1},l_{2},l_{3})$ can be extended continuously by constant functions to be a function $\widetilde{\theta}_{i}(l)$ on $J^{3}$.
\end{lemma}
Inspired by Lemma \ref{ex.cont.3}, we have the following observation. The proof is analogous to the proof of Lemma \ref{ex.cont.3}, we omit it here.
\begin{lemma}\label{ex.cont.2}
Let $l_{i}$, $\theta_{i}$ ($i=1,2,3$) be the edge length and opposite inner angle of a geometric triangle such that $\theta_{i}$ facing $l_{i}$. The following statements hold:
\begin{itemize}
\item[$(a)$] Suppose $l_{3}$ is a positive constant. The cosine law function $\theta_{i}(l_{1},l_{2})$ can be extended continuously by constant functions to be $\widetilde{\theta}_{i}(l_{1},l_{2})$ on $J^{2}$.
\item[$(b)$] Suppose $l_{2}$, $l_{3}$ are positive constants. The cosine law function $\theta_{i}(l_{1})$ can be extended continuously by constant functions to be $\widetilde{\theta}_{i}(l_{1})$ on $J^{1}$.
\end{itemize}
\end{lemma}
Replacing $\theta_{i}$ in $\omega^{E}_{\phi,s}$, $\omega^{H}_{\phi,s}$, $\omega^{S}_{\phi,s}$, $\omega^{H}_{\psi,s}$ with $\widetilde{\theta}_{i}$, we obtain four new differential forms $\widetilde{\omega}^{E}_{\phi,s}$, $\widetilde{\omega}^{H}_{\phi,s}$, $\widetilde{\omega}^{S}_{\phi,s}$ and $\widetilde{\omega}^{H}_{\psi,s}$.
\begin{lemma}\label{closeness}
Differential forms $\widetilde{\omega}^{E}_{\phi,s}$, $\widetilde{\omega}^{H}_{\phi,s}$, $\widetilde{\omega}^{S}_{\phi,s}$ and $\widetilde{\omega}^{H}_{\psi,s}$ are closed on $J^{s}$.
\end{lemma}
\begin{proof}
First, we want to show that $\widetilde{\omega}^{E}_{\phi,s}$ is closed. By Proposition \ref{B1} , we need to verify the following three propositions.
\begin{itemize}
\item[$(i)$] Differential form $\widetilde{\omega}^{E}_{\phi,s}$ is continuous on $J^{s}$ and closed on $\Omega_{s}^{E}$.
\item[$(ii)$] Differential form $\widetilde{\omega}^{E}_{\phi,s}$ is closed on each connected component of $J^{s}\setminus\Omega_{s}^{E}$.
\item[$(iii)$] The boundary of $\Omega_{s}^{E}$ is a $(s-1)$-dimensional smooth submanifold in $J^{s}$.
\end{itemize}
By Lemma \ref{ex.cont.3}, Lemma \ref{ex.cont.2}, $(i)$ is proved. Proposition $(iii)$ follows from the construction of $\Omega_{s}^{E}$. Proposition $(ii)$ is trivial because $\widetilde{\theta}_{i}$ is constant on each connected components of $J^{s}\setminus\Omega_{s}^{E}$. The proof of closeness of $\widetilde{\omega}^{H}_{\phi,s}$, $\widetilde{\omega}^{S}_{\phi,s}$ and $\widetilde{\omega}^{H}_{\psi,s}$ is quite similar to the proof of the closeness of $\widetilde{\omega}^{E}_{\phi,s}$.
\end{proof}
Similarly, replacing $\omega^{E}_{\phi,s}$, $\omega^{H}_{\psi,s}$, $\omega^{S}_{\phi,s}$ and $\omega^{H}_{\phi,s}$ in $F^{E}_{s}$, $F^{H}_{s}$, $F^{S}_{s}$ and $G^{H}_{s}$ with $\widetilde{\omega}^{E}_{\phi,s}$, $\widetilde{\omega}^{H}_{\psi,s}$, $\widetilde{\omega}^{S}_{\phi,s}$ and $\widetilde{\omega}^{H}_{\phi,s}$, we obtain four new energy functionals $\widetilde{F}^{E}_{s}$, $\widetilde{F}^{H}_{s}$, $\widetilde{F}^{S}_{s}$ and $\widetilde{G}^{H}_{s}$. By Lemma \ref{closeness}, these four energy functionals are well-defined on $\xi(J^{s})$ and $\gamma(J^{s})$. Now we want to prove the convexity of these four energy functionals.
\begin{lemma}\label{convex}
The energy functionals $\widetilde{F}^{E}_{s}$, $\widetilde{F}^{H}_{s}$, $\widetilde{F}^{S}_{s}$ are convex on $\xi(J^{s})$ when $s=1,2,3$ and the energy functional $\widetilde{G}^{H}_{s}$ are convex on $\gamma(J^{s})$ when $s=1,2$.
\end{lemma}
\begin{proof}
First, we want to show that $\widetilde{F}^{E}_{s}$ is convex on $\xi(J^{s})$. Owing to Proposition \ref{B2}, we need to check the following three propositions.
\begin{itemize}
\item[$(i)$] Differential form $\widetilde{\omega}^{E}_{h,s}$ is continuous on $J^{s}$.
\item[$(ii)$] The boundary of $\xi(\Omega^{E}_{s})$ is a real analytic codimension-$1$ submanifold in $\xi(J^{s})$.
\item[$(iii)$] Function $\widetilde{F}^{E}_{s}$ is locally convex on $\xi(\Omega^{E}_{s})$ and each connected component of $\xi(J^{s})\setminus \overline{\xi(\Omega_{s}^{E})}$.
\end{itemize}
By Lemma \ref{ex.cont.3}, Lemma \ref{ex.cont.2}, $(i)$ is proved. Combining the definition of $\Omega^{E}_{s}$ and Lemma \ref{homeomorphism}, $(ii)$ is established. Lemma \ref{locconv} yields the locally convexity of $\widetilde{F}^{E}_{s}$ on $\xi(\Omega^{E}_{s})$ directly. It is easy to see
\[\xi(J^{s})\setminus \overline{\xi(\Omega^{E}_{s})}=\interior\xi(\Omega_{12})\cup\interior \xi(\Omega_{23})\cup\interior\xi(\Omega_{13}).\]
From
\[\frac{\partial\widetilde{F}^{E}_{s}}{\partial u_{i}}=\int^{\widetilde{\theta}_{i}}_{\pi/2}\sin^{h}t\mathrm{d}t,\]
we know that $\widetilde{F}^{E}_{s}$ is a linear function on $\interior\xi(\Omega_{ij})$. Thus $(iii)$ is proved. Similar deduction gives the convexity of $\widetilde{F}^{H}_{s}$, $\widetilde{F}^{S}_{s}$ and $\widetilde{G}^{H}_{s}$.
\end{proof}

\section{Energy functionals on the moduli spaces of circle packing metrics}\label{4}
Following the idea of the previous section, we will construct the convex functionals on the moduli space of triangles for circle packing metric. We still begin by introducing a continuous map defined on the moduli space of triangles. Let $\tau$  be a triangle in a bordered triangulated surface. If all vertices of $\tau$ are interior vertices, denote $r=(r_{1},r_{2},r_{3})\in\mathbb{R}^{3}_{+}$ as the circle packing metric on $\tau$. Then the continuous map $\gamma$ is defined as
\begin{equation}\label{g}
g(r)=\left(g_{h}(r_{1}),g_{h}(r_{2}),g_{h}(r_{3})\right).
\end{equation}
where
\[g_{h}(t)=\left\{
             \begin{aligned}
             &\frac{t^{h}}{h},  h\neq0, &\tau\ \text{is a Euclidean triangle}, \\
             &\ln{t},  h=0, &\tau\ \text{is a Euclidean triangle},\\
             &\int^{t}_{1}\sinh^{h-1}\frac{x}{2}\mathrm{d}x, &\tau\ \text{is a hyperbolic triangle}.
             \end{aligned}
\right.\]
Similarly, the continuous map $g$ can be defined on the moduli space of the triangles which have one or two interior vertices. Because $g_{h}(t)$ is monotonic, the continuous map $g$ is homeomorphic to its image.

Now it is ready to introduce convex functionals defined on the moduli space of triangles. Define line integral function
\begin{equation}\label{cicfunc}
             C^{E}_{s}(u):=-\int^{u}\eta^{E}_{s}, u\in g(\mathbb{R}^{3}_{+}),\quad
             C^{H}_{s}(u):=-\int^{u}\eta^{H}_{s}, u\in g(\mathbb{R}^{3}_{+}).
\end{equation}
By Poincar\'{e} Lemma and Lemma \ref{path}, above functions are well-defined.
\begin{lemma}\label{cpcon}
Let $C^{E}_{s}(u)$, $C^{H}_{s}(u)$ be defined in (\ref{cicfunc}). The following statements hold:
\begin{itemize}
\item[$(a)$]The integral $C^{E}_{3}(u)$ is convex, $C^{E}_{s}(u)$ is strictly convex when $s=1,2$.
\item[$(b)$]The integral $C^{H}_{s}(u)$ is strictly convex for $s=1,2,3$.
\end{itemize}
\end{lemma}
\begin{proof}
Some computations yield that $C^{E}_{s}(u)$, $C^{H}_{s}(u)$ is smooth. It suffices to verify the positive-definiteness of the Hessian matrix of $C^{E}_{s}(u)$, $C^{H}_{s}(u)$. First we prove $(a)$. By (\ref{g}), we have
\[
\eta^{E}_{3}=\sum_{i=1}^{3}\frac{\int_{\pi/2}^{\theta_{i}}\tan^{h}{\frac{t}{2}}\mathrm{d}t}{r_{i}^{-h+1}}\mathrm{d}r_{i}
=\sum_{i=1}^{3}\int_{\pi/2}^{\theta_{i}}\tan^{h}{\frac{t}{2}}\mathrm{d}t\mathrm{d}u_{i}.\]
A regular computation gives that
\[\begin{aligned}
\frac{\partial^{2}C^{E}_{3}}{\partial u_{i}\partial u_{j}}&=-\tan^{h}\frac{\theta_{i}}{2}\frac{\partial\theta_{i}}{\partial u_{j}}\\
&=-\tan^{h}\frac{\theta_{i}}{2}\frac{\partial\theta_{i}/\partial r_{j}}{g'_{h}(r_{j})}\\
&=-(\tan\frac{\theta_{i}}{2}r_{i})^{h}(r_{i}r_{j})^{-h}r_{j}\frac{\partial\theta_{i}}{\partial r_{j}}.
\end{aligned}\]
By Proposition \ref{tan1},  $\tan\frac{\theta_{i}}{2}r_{i}$ is a positive number independent of the indices. Hence the Hessian matrix of $C^{E}_{3}(u)$ is congruent to $[a_{st}]$ where
\[a_{ij}=-r_{j}\frac{\partial\theta_{i}}{\partial r_{j}}.\]
From Lemma \ref{2.1}, we know
\begin{equation}\label{aii}\begin{aligned}
a_{ii}=-r_{i}\left(\frac{\partial\theta_{i}}{\partial l_{j}}+\frac{\partial\theta_{i}}{\partial l_{k}}\right)
=r_{i}\frac{\partial\theta_{i}}{\partial l_{i}}(\cos\theta_{k}+\cos\theta_{j})>0.
\end{aligned}\end{equation}
Similarly,
\begin{equation}\label{aij}\begin{aligned}
a_{ij}=-r_{j}\left(\frac{\partial\theta_{i}}{\partial l_{i}}+\frac{\partial\theta_{i}}{\partial l_{k}}\right)
=r_{j}\frac{\partial\theta_{i}}{\partial l_{i}}(\cos\theta_{j}-1)<0.
\end{aligned}\end{equation}
Hence,
\begin{equation}\label{diag}\begin{aligned}
|a_{ii}|-|a_{ji}|-|a_{ki}|&=a_{ii}+a_{ji}+a_{ki}\\
&=-r_{i}\frac{\partial(\theta_{i}+\theta_{j}+\theta_{k})}{\partial r_{i}}=0.
\end{aligned}\end{equation}
Thus the matrix $[a_{st}]$ is positive semi-definite due to the diagonal dominate property. Thus the Hessian matrix of $C^{E}_{3}(u)$ is positive semi-definite, which yields that $C^{E}_{3}(u)$ is convex. From (\ref{diag}) we also derive that
\[
|a_{ii}|-|a_{ji}|>0,
\]
which yields that the Hessian matrix of $C^{E}_{2}(u)$ is positive definite. Thus $C^{E}_{2}(u)$ is strictly convex. From (\ref{aii}) we know
\[\frac{\mathrm{d}^{2}C^{E}_{1}(u)}{\mathrm{d}u^{2}}>0,\]
which yields that $C^{E}_{1}(u)$ is strictly convex.

Next we prove $(b)$. An analogous analysis derives that the Hessian matrix of $C^{H}_{3}(u)$ are congruent to $[b_{st}]$ where
\begin{equation}\label{bij}b_{ij}=-\sinh r_{j}\frac{\partial\theta_{i}}{\partial r_{j}}.\end{equation}
Similar deduction to (\ref{aii}), (\ref{aij}) gives $b_{ii}>0, b_{ij}<0$.
Combining (\ref{bij}) and Lemma \ref{chow-luo}, we obtain
\[\begin{aligned}
|b_{ii}|-|b_{ji}|-|b_{ki}|&=b_{ii}+b_{ji}+b_{ki}\\
&=-\sinh r_{i}\frac{\partial(\theta_{i}+\theta_{j}+\theta_{k})}{\partial r_{i}}>0\\
\end{aligned}\]
which yields the strict convexity of $C^{H}_{3}(u)$. A similar deduction to the proof of $(a)$ gives the strict convexity of $C^{H}_{1}(u)$, $C^{H}_{2}(u)$.
\end{proof}

\section{Proof of main results}\label{5}
Let us prove our main results via variational principle in this section. We begin by introducing some notations that will be used in the proofs. For any polyhedral metric $d$ on $(S,T)$, denote
\[l_{d}=\big(d(e_{1}),\cdots,d(e_{n})\big)\] as the coordinate of $d$, where $\{e_{i}\}^{n}_{i=1}$ are all interior edges of $(S,T)$. Denote $u_{d}=\xi(l_{d})$
as the $u$-coordinate of $d$, where $\xi$ is defined by (\ref{xi}). Suppose that $\tau$ is a triangle in $T$. We write the coordinate of interior edges of $\tau$ in $d$ as $l^{\tau}_{d}$, and set $u^{\tau}_{d}=\xi(l^{\tau}_{d})$.
For $s=1,2,3$, denote
\[T_{s}=\left\{\tau\in T\ |\ \#s\ \text{edges in}\ \tau\ \text{are interior edges}\right\}.\]

Next, we construct some functions to be used in the proofs. They are constructed under the assumption $\cup_{s=1}^{3}T_{s}\neq\emptyset$. Define a function $W^{E}:\xi_{h}(\mathbb{R}_{+})^{n}\rightarrow\mathbb{R}$ by
\[W^{E}(u_{d})=\sum\nolimits_{s=1}^{3}\sum\nolimits_{\tau\in T_{s}}\widetilde{F}^{E}_{s}(u^{\tau}_{d}).\]
For each interior edge $e_{i}$, there exist two triangles
$\tau_{1}^{i},\ \tau_{2}^{i}\in\cup_{s=1}^{3}T_{s}$
such that $\tau_{1}^{i}\cap\tau_{2}^{i}=e_{i}$. Hence we have
\[\begin{aligned}
\left.\frac{\partial W^{E}}{\partial u_{i}}\right|_{u_{d}}&=\frac{\partial\widetilde{F}^{E}_{s}(u^{\tau^{i}_{1}}_{d})}{\partial u_{i}}+\frac{\partial\widetilde{F}^{E}_{s'}(u^{\tau^{i}_{2}}_{d})}{\partial u_{i}}\\
&=\int^{\widetilde{\theta}_{1}^{i}}_{\pi/2}\sin^{h}t\mathrm{d}t+\int^{\widetilde{\theta}_{2}^{i}}_{\pi/2}\sin^{h}t\mathrm{d}t,
\end{aligned}\]
where $\theta_{1}^{i}$, $\theta_{2}^{i}$ are angles facing $e_{i}$ in $\tau_{1}^{i}$, $\tau_{2}^{i}$, respectively. Note that
\[\left.\frac{\partial W^{E}}{\partial u_{i}}\right|_{u_{d}}=\phi_{h}(e_{i}).\]
It follows that
\[\nabla W^{E}(u_{d})=\big(\phi_{h}(e_{1}),\cdots,\phi_{h}(e_{n})\big).\]
Suppose that $u_{d_{1}}$ and $u_{d_{2}}$ are $u$-coordinates of two distinct polyhedral metrics. Set
\[f^{E}(t)=W^{E}\left(tu_{d_{1}}+(1-t)u_{d_{2}}\right).\]
\begin{lemma}\label{near}
Suppose $d_{1}$, $d_{2}$ are two distinct polyhedral metrics on $(S,T)$. Then $f^{E}(t)$ is strictly convex for $t$ near $0$,$1$.
\end{lemma}
\begin{proof}
For each $\tau\in T_{s}$, we would like to consider
\[f^{\tau}_{s}(t)=\widetilde{F}^{E}_{s}\left(tu^{\tau}_{d_{1}}+(1-t)u^{\tau}_{d_{2}}\right).\]
It follows that
\[
f^{E}(t)=\sum\nolimits_{s=1}^{3}\sum\nolimits_{\tau\in T_{s}}f^{\tau}_{s}(t).
\]
For each $\tau\in T$, if $tu^{\tau}_{d_{1}}+(1-t)u^{\tau}_{d_{2}}\in\xi\left(\Omega^{E}_{s}\right)$, we have
\begin{equation}
\label{''}
{f^{\tau}_{s}}''(t)=\left\{
\begin{aligned}
&\left(u^{\tau}_{d_{1}}-u^{\tau}_{d_{2}}\right)^{\mathrm{T}}\Hessian\widetilde{F}^{E}_{s}\left(tu^{\tau}_{d_{1}}+(1-t)u^{\tau}_{d_{2}}\right)\left(u^{\tau}_{d_{1}}-u^{\tau}_{d_{2}}\right),&s=2,3,\\
&\left(u^{\tau}_{d_{1}}-u^{\tau}_{d_{2}}\right)^2\widetilde{F}^{E}_{s}{''}\left(tu^{\tau}_{d_{1}}+(1-t)u^{\tau}_{d_{2}}\right),&s=1.
\end{aligned}
\right.
\end{equation}
By Lemma \ref{locconv}, we have ${f^{\tau}_{s}}''(t)\geq0$ if $tu^{\tau}_{d_{1}}+(1-t)u^{\tau}_{d_{2}}\in\xi\left(\Omega^{E}_{s}\right)$.
Recall that $\xi$ is a homomorphic map from the moduli space of triangles to its image. So $u_{d_{1}}^{\tau}$, $u_{d_{2}}^{\tau}$ are interior points for each $\tau\in T$. Hence, for each $\tau\in T$, there exists $\varepsilon_{\tau}>0$ such that $tu^{\tau}_{d_{1}}+(1-t)u^{\tau}_{d_{2}}$ is $u$-coordinate of a Euclidean metric for any $t\in(0,\varepsilon_{\tau})\cup(1-\varepsilon_{\tau},1)$. Set $\varepsilon_{0}=\min\left\{\varepsilon_{\tau}\ |\ \tau\in T\right\}$. Hence, we have
\[{f^{E}}''(t)=\sum\nolimits_{s=1}^{3}\sum\nolimits_{\tau\in T_{s}}{f^{\tau}_{s}}''(t)\geq 0\]
for $t\in(0,\varepsilon_{0})\cup(1-\varepsilon_{0},1)$.
According to (\ref{''}) and the Hessian matrix of $F^{E}_{s}(u)$, we derive the following three claims.
\begin{itemize}
\item[$(i)$] For $\tau\in T_{1}$, ${f^{\tau}_{1}}''(t)=0$ for $t\in(0,\varepsilon_{\tau})\cup(1-\varepsilon_{\tau},1)$ if and only if $u_{d_{1}}^{\tau}=u_{d_{2}}^{\tau}$.
\item[$(ii)$] For $\tau\in T_{2}$, ${f^{\tau}_{2}}''(t)=0$ for $t\in(0,\varepsilon_{\tau})\cup(1-\varepsilon_{\tau},1)$ if and only if $u_{d_{1}}^{\tau}=u_{d_{2}}^{\tau}$.
\item[$(iii)$] For $\tau\in T_{3}$, ${f^{\tau}_{3}}''(t)=0$ for $t\in(0,\varepsilon_{\tau})\cup(1-\varepsilon_{\tau},1)$ if and only if $u_{d_{1}}^{\tau}=cu_{d_{2}}^{\tau}(c\neq0)$.

\end{itemize}

Note that polyhedral surface is connected. That means any two distinct polyhedral metrics $d_{1}$, $d_{2}$ cannot satisfy both $u_{d_{1}}^{\tau}=cu_{d_{2}}^{\tau}$ for $\tau\in T_{3}$ and $u_{d_{1}}^{\tau}=u_{d_{2}}^{\tau}$ for $\tau\in T_{1}\cup T_{2}$, which yields ${f^{E}}''(t)>0$ for any $t\in(0,\varepsilon_{0})\cup(1-\varepsilon_{0},1)$.
Then $f^{E}(t)$ is strictly convex on $(0,\varepsilon_{0})\cup(1-\varepsilon_{0},1)$.
\end{proof}
\begin{proof}[\textbf{Proof of Theorem \ref{mt1}}]
We only prove Theorem \ref{mt1}$(a)$. Theorem \ref{mt1}$(b)$,$(c)$ can be obtained by a similar analysis. If $\cup_{i=s}^{3}T_{s}=\emptyset$, Theorem \ref{mt1}$(a)$ holds obviously. So we suppose $\cup_{i=s}^{3}T_{s}\neq\emptyset$. First we prove $(a)$. Let $d_{1}$, $d_{2}$ be two Euclidean polyhedral metrics that share equal boundary value and $\phi_{h}$-curvature on the interior edges. It suffices to show that $d_{1}=d_{2}$. Let $a\in\mathbb{R}^{n}$ be the $\phi_{h}$-curvature of $d_{1}$, $d_{2}$ on interior edges, where each vector component of $a$ represents the $\phi_{h}$-curvature of an interior edge. Then
\[\widetilde{W}^{E}(u)=W^{E}(u)-a\cdot u\]
is convex and $u_{d_{1}}$, $u_{d_{2}}$ are critical points of $\widetilde{W}^{E}(u)$. Hence
\[\widetilde{W}^{E}\big(tu_{d_{1}}+(1-t)u_{d_{2}}\big)=W^{E}\big(tu_{d_{1}}+(1-t)u_{d_{2}}\big)-a\cdot\big(tu_{d_{1}}+(1-t)u_{d_{2}}\big)\]
is a constant. It follows that
\[f^{E}(t)=W^{E}\big(tu_{d_{1}}+(1-t)u_{d_{2}}\big)\]
is a linear function if $u_{d_{1}}\neq u_{d_{2}}$. Combining with Lemma \ref{near}, we know that $u_{d_{1}}=u_{d_{2}}$. Recall that $f_{1}$ is a homeomorphic map to its image. Thus $l_{d_{1}}=l_{d_{2}}$. That means $d_{1}$ and $d_{2}$ are identical on the interior edges of $(S,T)$. Then we complete the proof.
\end{proof}
The following lemma is a classical result for cyclic polygons, which was first proved by B. Stanko \cite{sta}.
\begin{lemma}\label{cyc}
Let $Q$ be a Euclidean, or hyperbolic, or spherical quadrilateral. Then $Q$ is cyclic if and only if the sum of opposite inner angles of $Q$ are equal.
\end{lemma}

\begin{proof}[\textbf{Proof of Corollary \ref{co}}]
Let $P$ be a Euclidean or hyperbolic cyclic polygon. By adding edges between appropriate pairs of vertices of $P$, it can be realized as a bordered triangulated surface, denoted as $(P,T)$. (see Figure \ref{fig2} for an example). Lemma \ref{cyc} yields that $\psi_{0}(e)=0$ for each interior edge $e$ of $(P,T)$. By Theorem \ref{mt1} $(a)$, $(b)$, the result holds immediately.
\end{proof}

\begin{proof}[\textbf{Proof of Theorem \ref{1.7}}]
If $\cup_{s=1}^{3}T_{s}=\emptyset$, Theorem \ref{1.7} holds obviously. So we suppose that $\cup_{s=1}^{3}T_{s}\neq\emptyset$. Let $d_{1}$ and $d_{2}$ be two hyperbolic polyhedral metrics that have the equal boundary value and $\phi_{h}$-curvature on the interior edges.  It suffices to show that $d_{1}=d_{2}$. Here we set $u_{d}=\gamma(l_{d})$, where $\gamma$ is defined by (\ref{xi}). Denote the number of interior edges by $n$. Then we define a function $V^{E}:\gamma_{h}(\mathbb{R})^{n}\rightarrow\mathbb{R}$ by
\[V^{E}(u_{d})=\sum\nolimits_{s=1}^{2}\sum\nolimits_{\tau\in T_{s}}\widetilde{G}^{E}_{s}(u^{\tau}_{d}).\]
Note that $(S,T)$ is stripped, which means $T_{3}=\emptyset$. Hence, we obtain
\[\nabla V^{E}(u_{d})=\big(\phi_{h}(e_{1}),\cdots,\phi_{h}(e_{n})\big).\]
We denote $b\in\mathbb{R}^{n}$ as the $\phi_{h}$-curvature of $d_{1}$, $d_{2}$ on interior edges. Then
\[\widetilde{V}^{E}(u)=V^{E}(u)-b\cdot u\]
is convex and $u_{d_{1}}$, $u_{d_{2}}$ are critical points of $V^{E}(u)$. By similar analysis as the proof of Theorem \ref{mt1} we conclude that $d_{1}=d_{2}$.
\end{proof}

Let us introduce some similar notions before we prove Theorem \ref{mt2}. For each circle packing metric $r$ on $(S,T)$, denote
$p_{r}=\big(r(v_{1}),\cdots,r(v_{n})\big)$
as the coordinate of $r$ on all interior vertices $\{v_{j}\}^{n}_{j=1}$ of $(S,T)$. Suppose that $\tau$ is a triangle in $T$. We write the coordinate of radius assignment of interior vertices of $\tau$ as $p^{\tau}_{r}$, and set $q^{\tau}_{r}=g(p^{\tau}_{r})$. One can recall Section \ref{4} for the definition of homeomorphism $g$. For $s=1,2,3$, we denote
\[T_{s}=\left\{\tau\in T\ \big|\ \#s\ \text{vertices in}\ \tau\ \text{are interior vertices}\right\}.\]
Define a function $U^{E}:g_{h}(\mathbb{R}_{+})^{|n|}\rightarrow\mathbb{R}$ by
\[U^{E}(q_{r})=\sum\nolimits_{s=1}^{3}\sum\nolimits_{\tau\in T_{s}}C^{E}_{s}(q^{\tau}_{r}).\]
Similarly, we can define $U^{H}(q_{r})$. Note that the triangulated surface is bordered. Hence $T_{1}\cup T_{2}\neq\emptyset$. By Lemma \ref{cpcon}, the following result is straightforward.
\begin{lemma}\label{cpconv}
Function $U^{E}(q_{r})$, $U^{H}(q_{r})$ are strictly convex on $g_{h}(\mathbb{R}_{+})^{|n|}$.
\end{lemma}

\begin{proof}[\textbf{Proof of Theorem \ref{mt2}}]
If all vertices are boundary vertices, Theorem \ref{mt2} holds obviously. So we assume that $\cup_{s=1}^{3}T_{s}\neq\emptyset$. By the construction of $U^{E}$ we obtain
\[\nabla U^{E}(q_{r})=\left(-\sum\nolimits_{i=1}^{m_{1}}\int_{\pi/2}^{\theta_{i}^{1}}\tan^{h}(t/2)\textit{d}t,\cdots,-\sum\nolimits_{i=1}^{m_{n}}\int_{\pi/2}^{\theta_{i}^{n}}\tan^{h}(t/2)\textit{d}t\right),\]
where $\theta_{1}^{j},\cdots,\theta_{m_{j}}^{j}$ are all inner angles at vertex $v_{j}$. Combining Lemma \ref{cpconv} and Lemma \ref{importantlemma}, we obtain a homeomorphic map from the coordinate of $r$ to its $k_{h}$-curvature. This map is injective, obviously. Therefore, if two Euclidean circle packing metrics have the equal $k_{h}$-curvature on the interior vertices, then there are equal on the interior vertices. This proves $(a)$. The proof of $(b)$ can be completed by using the same analysis to $U^{H}(q_{r})$.
\end{proof}
The following lemma is an elementary result from analysis. We omit the proof here.
\begin{lemma}\label{importantlemma}
If $\Omega$ is open convex in $\mathbb{R}^{n}$ and $W:\Omega\rightarrow\mathbb{R}$ is a smooth strictly convex function with positive definite Hessian matrix, then $\nabla W:\Omega\rightarrow\mathbb{R}^{n}$ is a smooth embedding.
\end{lemma}

\begin{proof}[\textbf{Proof of Corollary \ref{cocp}}]
Let $P$ be a triangle-type circle packing. There is a natural way to build the bordered triangulated surface $(S,T)_{P}$ that isomorphic to the contact graph of $P$. The vertex set of $(S,T)_{P}$ is the set of centers of circles in $P$, denote as $V_{P}$. The edge set of $(S,T)_{P}$ is the set of edges joining the centers of each pair of mutually tangent circles in $P$. For any radius assignment $R:V\to\mathbb{R}_{+}$ of $P$, one can induce a circle packing metric on $(S,T)_{P}$ by $r_{R}(v)=R(v)$ for $v\in V_{P}$. Suppose $R_{1}$, $R_{2}$ are two radius assignments of $P$ (keeping the contact graph of $P$ unchanged) that equal on each boundary circle. Then $r_{R_{1}}$, $r_{R_{2}}$ are two circle packing metrics on $(S,T)_{P}$ having equal boundary value. Because $P$ is a planar circle packing, $(S,T)_{P}$ is a bordered triangulated surface on Euclidean plane. Recall that $k_{0}(v)=2\pi-\sum_{i=1}^{m}\theta_{i}$, where $\theta_{1},\cdots,\theta_{1}$ are inner angles at $v$. Hence $k_{0}$-curvature of $r_{R_{1}}$, $r_{R_{2}}$ for each interior vertex of $(S,T)_{P}$ equals to zero. By Theorem\ref{mt2}, we obtain $r_{R_{1}}=r_{R_{2}}$. Thus $R_{1}=R_{2}$. This completes the proof.
\end{proof}

\begin{appendices}
\section{Some formulas in Trigonometry}\label{AppA}
This section is devoted to some results from trigonometry.
\begin{proposition}[Cosine law]\label{cosine}
Let $l_{1}$, $l_{2}$, $l_{3}$ be the edge lengths of a geometric triangle and $\theta_{1}$, $\theta_{2}$, $\theta_{3}$ be the opposite inner angles. Set $\{i,j,k\}=\{1,2,3\}$.
\begin{itemize}
\item[$(a)$]In Euclidean geometry, we have
\[\cos\theta_{i}=\frac{-l^2_{i}+l^2_{j}+l^2_{k}}{2l_{j}l_{k}},\quad 1=\frac{\cos\theta_{i}+\cos\theta_{j}\cos\theta_{k}}{\cos\theta_{j}\cos\theta_{k}}.\]
\item[$(b)$]In hyperbolic geometry, we have
\[\cos\theta_{i}=\frac{-\cosh l_{i}+\cosh l_{j}\cosh l_{k}}{\sinh l_{j}\sinh l_{k}},\quad\cosh l_{i}=\frac{\cos\theta_{i}+\cos\theta_{j}\cos\theta_{k}}{\sin\theta_{j}\sin\theta_{k}}.\]
\item[$(c)$]In spherical geometry, we have
\[\cos\theta_{i}=\frac{\cos l_{i}-\cos l_{j}\cos l_{k}}{\sin l_{j}\sin l_{k}},\quad\cos l_{i}=\frac{\cos\theta_{i}+\cos\theta_{j}\cos\theta_{k}}{\sin\theta_{j}\sin\theta_{k}}.\]
\end{itemize}
\end{proposition}
\begin{proposition}[Tangent law]\label{tan1}
Let $l_{1}$, $l_{2}$, $l_{3}$ be three edges of a triangle and $\theta_{1}$, $\theta_{2}$, $\theta_{3}$ are the corresponding inner angles. Set
\[h_{i}=\cos \alpha_{i}\coth\frac{l_{i}}{2},\ h'_{i}=r_{i}\tan\frac{\theta_{i}}{2},\ s_{i}=\cos \alpha_{i}\cot\frac{l_{i}}{2},\ s'_{i}=\sinh r_{i}\tan\frac{\theta_{i}}{2},\]
where $i\in\{1,2,3\}$ and $\alpha_{i}=\frac{1}{2}(\theta_{i}-\theta_{j}-\theta_{k})$, $r_{i}=\frac{1}{2}(l_{j}+l_{k}-l_{i})$. Then $h_{i}$, $h'_{i}$ (resp. $s_{i}$, $s'_{i}$) are positive numbers independent of indices in hyperbolic (resp. spherical) geometry.
\end{proposition}
\section{Basic results about closed forms}\label{AppB}
In this section, we give a simple introduction to some results on differential forms and convex functions constructed by differential forms. One refers to \cite{Luo3,Bo3} for more background.

A differential 1-form $\omega=\sum^{n}_{i=1}f_{i}(x)\mathrm{d}x_{i}$ is said to be continuous in an open set $U\subset\mathbb{R}^{n}$ if each $f_{i}(x)$ is continuous on $U$. A continuous 1-form is called closed if $\int_{\partial\gamma}\omega=0$ for each piecewise $C^{1}$-smooth null homologous loop $\gamma$ in $U$.

\begin{proposition}\label{B1}
Suppose $X$ is an open set in $\mathbb{R}^{n}$ and $A\subset X$ is an open subset bounded by a smooth $(n-1)$-dimensional submanifold in $X$. If $\omega=\sum^{n}_{i=1}f_{i}(x)\mathrm{d}x_{i}$ is a continuous 1-form on $X$ such that $\omega|_{A}$ and $\omega|_{X-\overline{A}}$ are closed where $\overline{A}$ is the closure of $A$ in $X$, then $\omega$ is closed in $X$.
\end{proposition}
\begin{proposition}\label{B2}
Suppose $X\subset\mathbb{R}^{n}$ is an open convex set and $A\subset X$ is an open subset of $X$ bounded by a codimension-1 real analytic submanifold in $X$. If $\omega=\sum^{n}_{i=1}f_{i}(x)\mathrm{d}x_{i}$ is a continuous closed 1-form on $X$ such that $F(x)=\int^{x}_{a}\omega$ is locally convex in $A$ and in $X-\overline{A}$, then $F(x)$ is convex in $X$.
\end{proposition}
\end{appendices}
\section{Acknowledgement}
The authors would like to thank  Ze Zhou, for his encouragement and helpful discussions. They also would like to thank NSF of China (No.11631010) and Postgraduate Scientific Research Innovation Project of Hunan Province  (CX20210397) for financial support.

\Addresses

\begin{thebibliography}{99}
\bibitem{An} E. M. Andreev, \emph{On convex polyhedra in Lobachevskii spaces}, Mat. Sb. (N.S.), 81(123):3 (1970), 445--478; Math. USSR-Sb. 10:3 (1970), 413--440.
\bibitem{Ba} T. Ba, Z. Zhou, \emph{ Rigidity of polyhedral surfaces with finite boundary components}, Sciencepaper Online, \url{http://www.paper.edu.cn/releasepaper/content/202104-130}.
\bibitem{Bo3} A. I. Bobenko, U. Pinkall, B. A. Springborn, \emph{Discrete conformal maps and ideal hyperbolic polyhedra,} Geom. Topol. 19 (2015), 2155--2215.
\bibitem{Bo1} A. I. Bobenko, B. A. Springborn, \emph{Variational principles for circle patterns and Koebe's theorem,} Trans. Amer. Math. Soc. 356 (2004), 659--689.
\bibitem{Bo2} A. I. Bobenko, B. A. Springborn, \emph{A discrete Laplace-Beltrami operator for simplicial surfaces,} Discrete Comput. Geom. 38 (2007), 740--756.
\bibitem{chow} B. Chow, F. Luo,  \emph{Combinatorial Ricci flows on surfaces,} J. Differential Geom. 63 (2003), 97--129.
\bibitem{Co}  Y. Colin de Verdiere, \emph{Un principe variationnel pour les empilements de cercles,} Invent. Math. 104 (1991), 655--669.
\bibitem{Con} R. Connelly, \emph{Rigidity of packings,} European J. Combin. 29 (2008), 1862--1871.
\bibitem{L.D.G} J. Dai, X.D. Gu, F. Luo, \emph{Variational principles for discrete surfaces,} Advanced Lectures in Mathematics 4, Higher Education Press, Beijing, 2008.
\bibitem{Ge1} H. Ge, B. Hua, Z. Zhou, \emph{Circle patterns on surfaces of finite topological type}, Amer. J. Math. 143 (2021), 1397--1430.
\bibitem{Ge2} H. Ge, B. Hua, Z. Zhou, \emph{Combinatorial Ricci flows for ideal circle patterns}, Adv. Math. 383 (2021), Paper No. 107698.
\bibitem{Glick} D. Glickenstein, \emph{Discrete conformal variations and scalar curvature on piecewise flat two and three dimensional manifolds}, J. Differential Geom. 87 (2011), 201--238.
\bibitem{GT} D. Glickenstein, J. Thomas, \emph{Duality structures and discrete conformal variations of piecewise constant curvature surfaces,} Adv. Math. 320 (2017), 250--278.
\bibitem{Gu2} X. Gu, R. Guo, F. Luo, J. Sun, T. Wu, \emph{A discrete uniformization theorem for polyhedral surfaces II,} J. Differential Geom. 109 (2018), 431--466.
\bibitem{Gu1} X. Gu, F. Luo, J. Sun, T. Wu, \emph{A discrete uniformization theorem for polyhedral surfaces,} J. Differential Geom. 109 (2018), 223--256.
\bibitem{Guo} R. Guo, \emph{Local rigidity of inversive distance circle packing}, Trans. Amer. Math. Soc. 363 (2011), 4757--4776.
\bibitem{Luo2} R. Guo, F. Luo, \emph{Rigidity of polyhedral surfaces, II}, Geom. Topol. 13 (2009), no. 4, 1265--1312.
\bibitem{Ko} H. Kourimska, L. Skuppin, B. Springborn, \emph{A variational principle for cyclic polygons with prescribed edge lengths}, Advances in discrete differential geometry, (2016) 177--195.
\bibitem{Le} G. Leibon, \emph{Characterizing the Delaunay decompositions of compact hyperbolic surfaces,} Geom. Topol. 6 (2002), no. 1, 361--391.
\bibitem{Luoarxiv} F. Luo, \emph{Rigidity of polyhedral surfaces}, arXiv:math.GT/0612714.
\bibitem{Luo3} F. Luo, \emph{Rigidity of polyhedral surfaces, III}, Geom. Topol. 15 (2011), 2299--2319.
\bibitem{Luo1} F. Luo, \emph{Rigidity of polyhedral surfaces, I}, J. Differential Geom. 96 (2014), 241--302.
\bibitem{Pe} R.C. Penner, \emph{The decorated Teichm\"{u}ller space of punctured surfaces,} Comm. Math. Phys. 113 (1987), 299--339.
\bibitem{Pi} I. Pinelis, \emph{Cyclic polygons with given edge lengths: Existence and uniqueness}, J. Geom. 82 (2005), 156--171.
\bibitem{Pin} U. Pinkall, K. Polthier, \emph{Computing discrete minimal surfaces and their conjugates}, Experiment. Math. 2 (1993), 15--36.
\bibitem{Ri1} I. Rivin, \emph{Euclidean structures on simplicial surfaces and hyperbolic volume,} Ann. of Math. 139 (1994), 553--580.
\bibitem{Ri2} I. Rivin, \emph{A characterization of ideal polyhedra in hyperbolic 3-space,} Ann. of Math. 143 (1996), 51--70.
\bibitem{Sc} J.M. Schlenker, \emph{Small deformations of polygons and polyhedra,} Trans. Am. Math. Soc. 359(2007), 2155--2189.
\bibitem{sta} B. Stanko, \emph{Zur Begründung der elementaren Inhaltslehre in der hyperbolishchen Ebene}, Math. Ann. 180 (1969) 256--268.
\bibitem{Ste} K. Stephenson, \emph{Introduction to circle packing: The theory of discrete analytic functions,} Cambridge University Press, Cambridge, 2005.
\bibitem{Thu} W. Thurston, \emph{Geometry and topology of 3-manifolds}, Princeton lecture notes, 1976.
\bibitem{Xu} X. Xu, \emph{Rigidity of inversive distance circle packings revisited}, Adv. Math. 332 (2018), 476--509.
\bibitem{Xu2} X. Xu, \emph{A new proof of Bowers-Stephenson conjecture}, Math. Res. Lett. 28 (2021), 1283--1306.
\bibitem{Xu3} X. Xu, \emph{Rigidity and deformation of discrete conformal structures on polyhedral surfaces}, preprint,
    \url{https://arxiv.org/abs/2103.05272}.
\bibitem{Zeng} W. Zeng, R. Guo, F. Luo, and X. Gu, \emph{Discrete heat kernel determines discrete riemannian metric}, Graph. Models 74 (2012), 121--129.
\bibitem{Zhou1} Z. Zhou,\emph{Circle patterns with obtuse exterior intersection angles}, preprint,
    \url{https://arxiv.org/abs/1703.01768}.
\bibitem{Zhou2} Z. Zhou, \emph{Producing circle patterns via configurations}, preprint,
    \url{https://arxiv.org/abs/2010.13076}.
\end{thebibliography}
\end{document}